\documentclass{amsart}
\usepackage{amsfonts}
\usepackage{amsmath,amssymb}
\usepackage{amsthm}
\usepackage{amscd}
\usepackage{graphics}
\usepackage{graphicx}

\theoremstyle{definition}{
\newtheorem{Def}{{\rm Definition}}
\newtheorem{Ex}{{\rm Example}}
\newtheorem{Rem}{{\rm Remark}}

}
\theoremstyle{plain}
{
\newtheorem{Cor}{Corollary}
\newtheorem{Prop}{Proposition}
\newtheorem{Thm}{Theorem}
\newtheorem{MainThm}{Main Theorem}
\newtheorem{MainCor}{Main Corollary}

}

\begin{document}
\title[Manifolds admitting no special generic maps and their cohomologies]{Closed manifolds admitting no special generic maps whose codimension are negative and their cohomology rings}
\author{Naoki Kitazawa}
\address{Institute of Mathematics for Industry, Kyushu University, 744 Motooka, Nishi-ku Fukuoka 819-0395, Japan\\
 TEL (Office): +81-92-802-4402 \\
 FAX (Office): +81-92-802-4405 \\
}
\email{n-kitazawa@imi.kyushu-u.ac.jp}
\urladdr{https://naokikitazawa.github.io/NaokiKitazawa.html}
\maketitle
\begin{abstract}

{\it Special generic} maps are higher dimensional versions of Morse functions with exactly two singular points, characterizing spheres topologically except $4$-dimensional cases: in these cases unit spheres are characterized.
Canonical projections of unit spheres are special generic. In suitable cases, it is easy to construct special generic maps on manifolds represented as connected sums of products of spheres for example. It is an interesting fact that these maps restrict the topologies and the differentiable structures admitting them strictly in various cases. For example, {\it exotic} spheres, which are not diffeomorphic to standard spheres, admit no special generic maps into some Euclidean spaces in considerable cases.

In general, it is difficult to find (families of) manifolds admitting no such maps of suitable classes. The present paper concerns a new result on this work where cohomology rings of the manifolds are key objects. 
We can see that several important manifolds such as closed symplectic manifolds and real projective spaces admit no special generic maps into any connected non-closed manifold in considerable cases for example.

\end{abstract}


\maketitle
\section{Introduction.}
\label{sec:1}

{\it Special generic} maps are higher dimensional versions of Morse functions with exactly two singular points, characterizing spheres topologically except $4$-dimensional cases: in these cases standard spheres are characterized as smooth manifolds. They have been attractive objects in understanding topologies and differentiable structures of manifolds in geometric ways, which is a fundamental and important study in geometry of manifolds.

\subsection{Notation on differentiable maps and bundles.}
\label{subsec:1.1}
Throughout the present paper, manifolds and maps between manifolds are smooth (of class $C^{\infty}$). Diffeomorphisms on manifolds are always assumed to be smooth. The {\it diffeomorphism group} of a manifold is the group of all diffeomorphisms on the manifold. The structure groups of bundles whose fibers are manifolds are assumed to be subgroups of the diffeomorphism groups or the bundles are {\it smooth} unless otherwise stated. 
Moreover, connected sums of smooth manifolds are considered in the smooth category.

${\mathbb{R}}^k$ denotes the $k$-dimensional Euclidean space and we regard this as a natural smooth manifold. This is also regarded as a Riemannian manifold with the standard Euclidean metric and $||x|| \geq 0$ denotes the metric between $x \in {\mathbb{R}}^k$ and the origin $0 \in {\mathbb{R}}^k$.
$\mathbb{R}$ is for ${\mathbb{R}}^1$ and the real number field.
$\mathbb{Z} \subset \mathbb{R}$ denotes the integer ring. We have a natural closed submanifold $S^k:=\{x \mid ||x||=1\} \subset {\mathbb{R}}^{k+1}$ for $k \geq 0$. This is the $k$-dimensional unit sphere. This is a two-point set endowed with the discrete topology for $k=0$ and a $k$-dimensional compact and connected submanifold with no boundary.
 A smooth manifold homeomorphic to a unit sphere is called a {\it homotopy sphere}. If it is (not) diffeomorphic to any unit sphere, then it is called a {\it standard} (resp. an {\it exotic}) sphere.
  The existence of exotic spheres is well-known for cases where the dimensions are greater than $6$. It is well-known that for cases where the dimensions are $1,2,3,5$, the non-existence is well-known. 
 We have another natural closed submanifold $D^k:=\{x \mid ||x|| \leq 1\} \subset {\mathbb{R}}^k$ for $k \geq 1$. This is the $k$-dimensional unit disk. This is a $k$-dimensional compact and connected submanifold with no boundary.   
A {\it linear} bundle is a smooth bundle whose fiber is regarded as a unit sphere or a unit disk in a Euclidean space and whose structure group acts linearly in a canonical way on the fiber.
A {\it singular} point $p \in X$ of a smooth map $c:X \rightarrow Y$ is a point at which the rank of the differential $dc$ is smaller than both the dimensions $\dim X$ and $\dim Y$. We call the set $S(c)$ of all singular points the {\it singular set} of $c$. We call $c(S(c))$ the {\it singular value set} of $c$. We call $Y-c(S(c))$ the {\it regular value set} of $c$. A {\it singular {\rm (}regular{\rm )} value} is a point in the singular (resp. regular) value set of the map.

\subsection{The definition of a special generic map, the topologies and the differentiable structures of manifolds admitting special generic maps, meanings in algebraic topology and differential topology of manifolds and Main Theorems.}
\label{subsec:1.2}
Let $m>n \geq 1$ be integers. A smooth map from an $m$-dimensional smooth manifold with no boundary into an $n$-dimensional smooth manifold with no boundary is said to be a {\it special generic} map if at each singular point $p$, the map is represented as
$$(x_1, \cdots, x_m) \mapsto (x_1,\cdots,x_{n-1},\sum_{k=n}^{m}{x_k}^2)$$
 for suitable coordinates. The restriction map to the singular set is an immersion.

We can define for the cases $m=n$. However, we concentrate on cases where $m>n$ hold.

We can define {\it fold} maps as higher dimensional versions of Morse functions similarly and the class of special generic maps is a proper subclass of that of fold maps. See \cite{golubitskyguillemin} for introductory and systematic explanations for example. As studies of the author, see also \cite{kitazawa, kitazawa2, kitazawa3} for example. 
However, general fold maps are not studied in the present paper.

 
Canonical projections of unit spheres are special generic. Every homotopy sphere of dimension greater than $2$
 except $4$-dimensional exotic spheres, which are still undiscovered, admits a special generic map into the plane whose restriction to the singular set is an embedding and whose singular value set is an embedded circle (\cite{saeki}).
For integers $m>n \geq 1$, on an $m$-dimensional manifold represented as a connected sum of manifolds represented as products of two standard spheres such that the dimension of either of the two sphere for each manifold is smaller than $n$, we can obtain a special generic map into ${\mathbb{R}}^n$ (we will present in Example \ref{ex:1}). We introduce interesting facts from the viewpoint of differential topology in \cite{saeki,  saeki2, saekisakuma,saekisakuma2, wrazidlo} for example. 
As studies before these recent ones, \cite{burletderham, calabi, furuyaporto, sakuma, sakuma2} are also related important studies.

\begin{Thm}[\cite{calabi,saeki,saeki2}]
\label{thm:1}
Exotic homotopy spheres of dimension $m>3$ do not admit special generic maps into ${\mathbb{R}}^{m-3}$, ${\mathbb{R}}^{m-2}$ and ${\mathbb{R}}^{m-1}$. 
\end{Thm}
\begin{Thm}[\cite{wrazidlo}]
\label{thm:2}
$7$-dimensional oriented homotopy spheres of $14$ types of all $28$ types do not admit special generic maps into ${\mathbb{R}}^3$.   
\end{Thm}
\cite{eellskuiper} is, for example, on classical theory of $7$-dimensional homotopy spheres other than \cite{milnor}.
\begin{Thm}[\cite{saeki}]
\label{thm:3}
A closed and connected manifold admits a special generic map into the plane if and only if it is diffeomorphic to one of the following manifolds or represented as a connected sum of the following manifold.
\begin{enumerate}
\item A homotopy sphere of dimension greater than or equal to $2$ which is not a $4$-dimensional exotic sphere.
\item The total space of a bundle whose fiber is a homotopy sphere not being a $4$-dimensional exotic sphere over the circle.
\end{enumerate}
\end{Thm}

\begin{Thm}[\cite{saeki, saekisakuma}]
\label{thm:4}
For an integer $3 \leq m \leq 5$, an $m$-dimensional closed and connected manifold has a free fundamental group and admits a special generic map into ${\mathbb{R}}^3$ if and only if it is diffeomorphic to one of the following manifolds or represented as a connected sum of the following manifolds.
\begin{enumerate}
\item The $m$-dimensional standard sphere.
\item The total space of a bundle whose fiber is diffeomorphic to $S^{m-2}$ over $S^2$.
\item The total space of a bundle whose fiber is diffeomorphic to $S^{m-1}$ over $S^1$.
\end{enumerate}

For an integer $m \geq 4$, $m $-dimensional closed and simply-connected manifolds admitting special generic maps into ${\mathbb{R}}^3$ are diffeomorphic to one of the following manifolds or represented as connected sums of the following manifolds.
\begin{enumerate}
\item Homotopy spheres admitting special generic maps into ${\mathbb{R}}^3$.
\item The total space of a bundle whose fiber is diffeomorphic to a homotopy sphere of dimension $m-2$ over $S^2$.
\end{enumerate}
\end{Thm}
For this, see also \cite{saekisakuma2}.
As another study, in \cite{nishioka}, Nishioka completely solved the so-called existence problem of special generic maps into Euclidean spaces on $5$-dimensional closed and simply-connected manifolds, which are completely classified in \cite{barden}. He has shown that such a manifold admits a special generic map into ${\mathbb{R}}^n$ for some integer $1 \leq n \leq 4$ if and only if it is a standard sphere or a manifold represented as a connected sum of total spaces of bundles whose fibers are diffeomorphic to $S^3$ over $S^2$.
In general, it is difficult to find (families of) manifolds admitting no special generic maps of suitable classes in general. As a related small proposition, the author has proved in \cite{kitazawa4, kitazawa5} first for example, that the cup product of given suitable two cohomology classes of a manifold admitting a special generic map into the Euclidean space of a fixed dimension vanishes. 
The present paper concerns a kind of advanced and systematic studies on cohomology rings of manifolds admitting special generic maps and ones admitting no special generic maps.

Throughout the present paper, as usual, the $j$-th (co)homology group of a topological space $X$ is denoted by $H_j(X;A)$ (resp. $H^j(X;A)$) and the cohomology ring of it is denoted by $H^{\ast}(X;A)$ where $A$ is the coefficient ring. In addition, the $j$-th homotopy group of $X$ is denoted by ${\pi}_k(X)$. The homomorphism between the $j$-th homology groups induced canonically from a map $i:X \rightarrow Y$ between two topological spaces is denoted by $i_{\ast}:H_j(X;A) \rightarrow H_j(Y;A)$ where $A$ is the coefficient ring. For the $j$-th homotopy groups, we use the same notation. For the $j$-th cohomology groups, we use $i^{\ast}:H^j(Y;A) \rightarrow H^j(X;A)$. The cup product of the $l>0$ elements of a family $\{a_j\}_{j=1}^l \subset H^{\ast}(X;A)$ of $l$ cohomology classes is denoted by ${\cup}_{j=1}^l a_j \in H^{\ast}(X;A)$. 

\begin{MainThm}
\label{mthm:1}
Let $M$ be a closed manifold of dimension $m>1$ and $1 \leq n<m$ and $l>0$ be integers.
Let $A$ be a commutative ring. If there exists a sequence $\{a_j\}_{j=1}^l \subset H^{\ast}(M;A)$ of cohomology classes such that the cup product ${\cup}_{j=1}^l a_j$ does not vanish, that the degree of each class is smaller than or equal to $m-n$ and
 that the sum of the degrees are greater than or equal to $n$, then $M$ admits no special generic maps into any $n$-dimensional connected non-closed manifold $N$ with no boundary.
\end{MainThm}

Other papers and preprints such as \cite{kitazawa6} refer to the present paper and proofs of Main Theorem \ref{mthm:1}. These proofs are essentially same as those in the present paper and \cite{kitazawa4, kitazawa5}. 

We have the following corollary.
\begin{MainCor}[Corollary \ref{cor:1}]
Let $m>1$ be an integer. An $m$-dimensioal manifold whose cohomology ring is isomorphic to that of the $m$-dimensional real projective space admits no special generic map into any $n$-dimensional non-closed connected manifold $N$ with no boundary for $1 \leq n \leq m-1$ where the coefficient ring is $\mathbb{Z}/2\mathbb{Z}$, which is of order $2$.
\end{MainCor}
The definition of a {\it c-symplectic} manifold is in the last section and we present another corollary.
\begin{MainCor}[Corollary \ref{cor:2}]
A closed c-symplectic manifold $M$ of dimension $m=2k>0$ admits no special generic maps into any $n$-dimensional connected non-closed manifold $N$ with no boundary satisfying $n<2k-1$.
\end{MainCor}

Furthermore, as a related study, in \cite{kitazawa5}, the author has proved the vanishing of the {\it triple Massey product} for a triplet of cohomology classes of a manifold admitting a special generic map of a suitable class for which we can define the triple Massey product. However, we do not discuss this in the present paper. See \cite{kraines} for {\it Massey products} for example.

Another main theorem is as follows.
 We present undefined several terminologies and notions in the last section. This is a theorem on Euclidean spaces into which given manifolds admit special generic maps. The condition ${\rm Sp}_{\geq n_0}$ implies the existence of special generic maps into ${\mathbb{R}}^{n}$ for $n_0 \leq n \leq m$ where $m$ and $n_0$ are the dimension of the manifold and a positive integer, respectively. The condition ${\rm CohP}_{A,m,n_0-1}$ implies the non-existence of special generic maps into ${\mathbb{R}}^{n}$ for $1 \leq n<n_0$ where $m$ and $n_0$ are as just before.
\begin{MainThm}
\label{mthm:2}
Let $A$ be a principal ideal domain {\rm (}having a unique identity element which is not the zero element{\rm )}. 
\begin{enumerate}
\item \label{mthm:2.1}
Let $M_1$ and $M_2$ be closed and connected manifolds of dimension $m>2$. Let $n_0$ be an integer greater than $1$ and smaller than $m$.
If $M_1$ and $M_2$ satisfy the conditions ${\rm Sp}_{\geq n_0}$ and ${\rm CohP}_{A,m,n_0-1}$, then a manifold represented as a connected sum of these manifolds also does.
\item \label{mthm:2.2}
 Let $M^{\prime}$ be a closed and connected manifold of dimension $m^{\prime}>2$ and for an integer ${n_0}^{\prime}$ greater than $1$ and smaller than $m^{\prime}$, let $M^{\prime}$ satisfy the conditions ${\rm Sp}_{\geq {n_0}^{\prime}}$ and ${\rm CohP}_{A,m^{\prime},{n_0}^{\prime}-1}$.
  We also assume that the homology group of $M^{\prime}$ is free where the coefficient ring is $A$.
 Let $n_0$ be a positive integer larger than ${n_0}^{\prime}$ and $F$ be a closed and connected manifold of dimension $n_0-{n_0}^{\prime}$ satisfying the following properties.
\begin{enumerate}
 \item
 \label{mthm:2.2.1}
  There exist a positive integer $l$ and a sequence $\{a_j\}_{j=1}^l \subset H^{\ast}(F;A)$ of cohomology classes such that the cup product ${\cup}_{j=1}^l a_j$ does not vanish, that the degree of each class is smaller than or equal to $m^{\prime}-{n_0}^{\prime}+1$ and that the sum of the degrees is equal to $n_0-{n_0}^{\prime}$.
 \item
  \label{mthm:2.2.2}
  The homology group of $F$ is free where the coefficient ring is $A$.
 \item 
  \label{mthm:2.2.3}
 We have an immersion of $F$ into ${\mathbb{R}}^{n_0}$ whose normal bundle is trivial.

\end{enumerate}
Let $m:=m^{\prime}+n_0-{n_0}^{\prime}$.
Then $M^{\prime} \times F$ satisfies the conditions ${\rm CohP}_{A,m,n_0-1}$ and ${\rm Sp}_{\geq n_0}$.
\end{enumerate}
\end{MainThm}

\subsection{The content of the present paper.}
\label{subsec:1.3}

The organization of the paper is as follows.
We review structures of special generic maps. In explaining about them, we use the {\it Reeb space} of a special generic map, which is defined as the space of all connected components of preimages.
 This is regarded as a compact manifold whose dimension is equal to that of the manifold of the target and which we can immerse there.
 We also present Example \ref{ex:1} as simplest examples of special generic maps.
The last section is devoted to the main ingredient including Main Theorems.

\section{Structures of special generic maps and a simplest example.}
\label{sec:2}
For a continuous map $c:X \rightarrow Y$ between topological spaces, we can define an equivalence relation ${\sim}_c$ on $X$ by the following rule: $p_1 {\sim}_c p_2$ if and only if $p_1$ and $p_2$ are in a same connected component of a preimage $c^{-1}(q)$ ($q \in Y$). We call the quotient space $W_c:=X/{{\sim}_c}$ the {\it Reeb space} of $c$. Let the quotient map be denoted by $q_c:X \rightarrow W_c$. We can define in a unique way a map $\bar{c}$ satisfying $f=\bar{c} \circ q_c$.
See also \cite{reeb} (as a classical and important study) for example. In general Reeb spaces are fundamental and strong tools in investigating manifolds via Morse functions, fold maps and more general good smooth maps.

\begin{Prop}[\cite{saeki} for example]
\label{prop:1}
\begin{enumerate}
\item
\label{prop:1.1}
 The Reeb space of a special generic map $f:M \rightarrow N$ from an $m$-dimensional closed and connected manifold $M$ into an $n$-dimensional manifold $N$ with no boundary satisfying $m>n$ is an $n$-dimensional manifold immersed into $N$ via $\bar{f}:W_f \rightarrow N$ and $q_f(S(f))=\partial W_f$ holds.
  Furthermore, around the boundary $\partial W_f$, the composition of the restriction of $q_f$ to the preimage of a small collar neighborhood with the canonical projection to the boundary $\partial W_f$ gives a linear bundle whose fiber is diffeomorphic to $D^{m-n+1}$.
   Moreover, the composition of the restriction of $q_f$ to the preimage of the complementary set of the interior of the small collar neighborhood gives a smooth bundle whose fiber is diffeomorphic to $S^{m-n}${\rm :} the bundle is linear in the case $m-n=1,2,3$ for example. Furthermore, in the PL or piecewise smooth category, we can take a compact manifold $W$ bounded by $M$ and collapsing to $W_f$ and  in the case $m-n=1,2,3$ for example, we can take $W$ as a smooth manifold. 
\item
\label{prop:1.2}
 For any immersion ${\bar{f}}_N$ of a compact and connected manifold $\bar{N}$ of dimension $n>0$ into an $n$-dimensional manifold $N$ with no boundary and any integer $m>n$, there exist an $m$-dimensional closed and connected manifold $M$ of dimension $m$ and a special generic map $f:M \rightarrow N$ satisfying $W_f=\bar{N}$ and $\bar{f}={\bar{f}}_N$. If $N$ is orientable, then $M$ can be taken as an orientable manifold.
\end{enumerate}
\end{Prop}

\begin{Ex}
\label{ex:1}
Let $M$ be a closed manifold of dimension $m>1$ represented as a connected sum of all manifolds in the family $\{S^{k_j} \times S^{m-k_j}\}$ of finitely many manifolds satisfying $1 \leq k_j <n$ where $1<n \leq m$ holds.
 This admits a special generic map into ${\mathbb{R}}^n$ such that the restriction to the singular set is an embedding and that the image is represented as a boundary connected sum of all manifolds in the family $\{S^{k_j} \times D^{n-k_j}\}$. Furthermore, the singular value set and the boundary of the image agree, and the preimage of each regular value in the image is diffeomorphic to $S^{m-n}$.
\end{Ex}
\section{The main theorems, their proofs and applications.}
The following proposition or essentially (almost) equivalent ones are shown in \cite{saeki} and \cite{kitazawa,kitazawa2,kitazawa3,kitazawa4,kitazawa5} for example.
\begin{Prop}
\label{prop:2}
In the situation of Proposition \ref{prop:1} {\rm (}\ref{prop:1.1}{\rm )} {\rm (}and {\rm (}\ref{prop:1.2}{\rm )}{\rm )}, assume that $N$ is connected and non-closed and let the inclusion map denoted by $i:M \rightarrow W$.
 For a suitable PL or piecewise smooth map giving a collapsing $r:W \rightarrow W_f$, $q_f=r \circ i$ holds. 
Let $A$ be a commutative group. Here, the induced morphisms $i_{\ast}:H_j(M;A) \rightarrow H_j(W_f;A)$, $i^{\ast}:H^j(W_f;A) \rightarrow H^j(M;A)$ and $i_{\ast}:{\pi}_j(M) \rightarrow {\pi}_j(W_f)$ are isomorphisms for $0 \leq j \leq m-n$. 
\end{Prop}
We explain about the main ingredient of the proof of the statement on the isomorphisms only.
\begin{proof}[The main ingredient of the proof of the statement on the isomorphisms.]
$W_f$ collapses to an ($n-1$)-dimensional polyhedron since it can be immersed into an $n$-dimensional connected and non-closed manifold $N$ and $W$ is obtained by attaching handles whose indices are greater than $m-n+1$ to $M \times \{1\} \subset M \times [0,1]$: we identify $M$ and $M \times \{0\}$ via the map $i_M(x):=(x,0)$ and $M \times [0,1]$ is regarded as a small collar neighborhood (in the category where we discuss).
\end{proof}


\begin{proof}[Proof of Main Theorem \ref{mthm:1}]
Suppose that $M$ admits a special generic map $f:M \rightarrow N$. We abuse the notation in Proposition \ref{prop:2} and apply this proposition.
There exists a unique sequence $\{b_j\}_{j=1}^l \subset H^{\ast}(W;A)$ of cohomology classes satisfying $a_j= i^{\ast}(b_j)$. We have ${\cup}_{j=1}^l a_j={\cup}_{j=1}^l i^{\ast}(b_j)= i^{\ast}({\cup}_{j=1}^l b_j)$
 and this is zero since $W$ collapses
 to an ($n-1$)-dimensional polyhedron. This contradicts the assumption. This completes the proof. 
\end{proof}

We present important examples related to the present problem.
\begin{Ex}
\label{ex:2}
Let $m>1$ be an integer. The $m$-dimensional torus admits no special generic maps into any $n$-dimensional connected non-closed manifold $N$ with no boundary for $1 \leq n \leq m-1$.
\cite{grossbergkarshon, masuda} are studies on manifolds obtained by considering finite iterations of constructing bundles whose fibers are the circle starting from the circle and some of manifolds of this class satisfy the assumption of our Main Theorem \ref{mthm:1} for arbitrary $1 \leq n \leq m-1$.
\end{Ex}
\begin{Rem}
By virtue of known algebraic topological and differential topological theory of special generic maps in the introduction and the previous sections and Proposition \ref{prop:2} for example,
 $S^1 \times S^1 \times S^1 \times S^1$ admits no special generic maps into ${\mathbb{R}}^3$. If it admits one, then the fundamental group of the Reeb space, which is diffeomorphic to a $3$-dimensional compact, connected and orientable manifold by Proposition \ref{prop:1}, and that of $S^1 \times S^1 \times S^1 \times S^1$ agree: however, this group is known to be never isomorphic to any fundamental group of any $3$-dimensional compact, connected and orientable manifold. 
Main Theorem \ref{mthm:1} produces another exposition on this.
\end{Rem}

A $2k$-dimensional closed manifold $X$ where $k>0$ is an integer is said to be {\it c-symplectic} if there exists a cohomology class $c_s \in H^2(X;\mathbb{R})$ such that the cup product ${\cup}_{j=1}^k c_s$, the $k$-th power of $c_s$, is not zero. See also \cite{kasuya, luptonoprea} for example.
\begin{Cor}
\label{cor:1}
A closed c-symplectic manifold $M$ of dimension $m=2k>0$ admits no special generic maps into any connected non-closed manifold $N$ of dimension $n<2k-1$ with no boundary.
\end{Cor}
For example, \cite{choimasudasuh,grossbergkarshon} are on algebraic topological and differential topological studies on some classes of (c-)symplectic closed manifolds.
For example, $S^2 \times S^2$ is regarded as a symplectic manifold and admits a special generic map into ${\mathbb{R}}^3$. On the other hand, the $2k$-dimensional torus is regarded as a symplectic manifold and admits no special generic maps into any $n$-dimensional connected non-closed manifold $N$ for $1 \leq n \leq 2k-1$ by Example \ref{ex:2}.
We also show another example. 
\begin{Ex}
\label{ex:3}
	Let $k>0$ be an integer. $S^k \times S^k \times S^k$ admits no special generic maps into any $2k$-dimensional connected non-closed manifold $N$ by Main Theorem \ref{mthm:1}. A manifold represented as a connected sum of three copies of $S^k \times S^{2k}$ admits a special generic map into ${\mathbb{R}}^{2k}$ by Example \ref{ex:1}. For these two, the $j$-th homology groups are as follows where the coefficient ring is $\mathbb{Z}$.
	\begin{itemize}
		\item They are isomorphic to $\mathbb{Z}$ for $j=0,3k$.
		\item They are isomorphic to $\mathbb{Z} \oplus \mathbb{Z} \oplus \mathbb{Z}$ for $j=k,2k$.
		\item They are trivial for $j \neq 0,k,2k,3k$.
	\end{itemize}
\end{Ex}
This shows explicitly that difference in cohomology rings affect the existence or non-existence of special generic maps.
\begin{Cor}
\label{cor:2}
Let $m>1$ be an integer. An $m$-dimensioal manifold $M$ whose cohomology ring is isomorphic to that of the $m$-dimensional real projective space admits no special generic maps into any $n$-dimensional connected non-closed manifold $N$ with no boundary
 for $1 \leq n \leq m-1$ where the coefficient ring is $\mathbb{Z}/2\mathbb{Z}$, the group of order $2$.
\end{Cor}
\begin{proof}
There exists a cohomology class $u_0 \in H^{1}(M;\mathbb{Z}/2\mathbb{Z})$ such that the cup product ${\cup}_{j=1}^m u_0$, the $m$-th power of $u_0$, is not zero.
 Main Theorem \ref{mthm:1} completes the proof.
\end{proof}
A weaker result with its proof in a different way is announced in \cite{wrazidlo2}. This is on manifolds homotopy equivalent to the $7$-dimensional real projective plane. 
After the first announcement of an earlier version of the present paper on https://arxiv.org/abs/2008.04226, a stronger result \cite{wrazidlo3} was announced. This says that a $7$-dimensional closed and connected manifold whose homology group is isomorphic to that of the $7$-dimensional real projective space admits no special generic maps into ${\mathbb{R}}^n$ for $n=1,2,3,4,5,6$ where the coefficient ring is the integer ring $\mathbb{Z}$.  
\begin{Def} 
\label{def:1}
Let $M$ be a closed manifold of dimension $m>1$. Let $n_0$ be a positive integer smaller than $m$.
$M$ is said to {\it satisfy the condition} ${\rm Sp}_{\geq n_0}$ if $M$ admits a special generic map into ${\mathbb{R}}^n$ for any integer $n_0 \leq n <m$.

\end{Def}
\begin{Ex}
\label{ex:4}
$M$ in Example \ref{ex:1} satisfies the condition ${\rm Sp}_{\geq \min\{k_j+1\}_{j \in J}}$ where $J$ denotes a finite set and is not empty in the situation of Example \ref{ex:1}.
\end{Ex}
The following notion is introduced motivated by Main Theorem \ref{mthm:1}.
\begin{Def}
\label{def:2}

Let $X$ be a closed and connected manifold of dimension $\dim X>1$ and $n_0$ be a non-negative integer smaller than $\dim X$.
Let $A$ be a principal ideal domain {\rm (}having a unique identity element which is not the zero element{\rm )}. 
If there exist a positive integer $l>0$ and a sequence $\{a_j\}_{j=1}^l \subset H^{\ast}(X;A)$ of cohomology classes such that the cup product ${\cup}_{j=1}^l a_j$ does not vanish, that the degree of each class is smaller than or equal to $\dim X-n_0$ and
 that the sum of the degrees are greater than or equal to $n_0$, then $X$ is said to {\it satisfy the condition} ${\rm CohP}_{A,\dim X,n_0}$.
 \end{Def}
\begin{proof}[Proof of Main Theorem \ref{mthm:2}]
We prove (\ref{mthm:2.1}).
 The condition ${\rm CohP}_{A,m,n_0-1}$ follows immediately. We can construct a special generic map on a closed manifold represented as a connected sum of closed and connected manifolds admitting special generic maps into ${\mathbb{R}}^n$ easily (for any $n$ satisfying $1 \leq n \leq m$). This is a fundamental argument in \cite{saeki} for example.
This completes the proof of (\ref{mthm:2.1}).

We show (\ref{mthm:2.2}).
By virtue of the first two conditions (\ref{mthm:2.2.1}) and (\ref{mthm:2.2.2}) and the condition on the homology group of $M^{\prime}$, by choosing suitable sequences of cohomology classes of $M^{\prime}$ and $F$, we can obtain a sequence of cohomology classes of $M^{\prime} \times F$ of a finite length where the coefficient ring is $A$ such that the degree of each class is smaller than or equal to $m^{\prime}-{n_0}^{\prime}+1$ and that the sum of the degrees are greater than or equal to $n_0-{n_0}^{\prime} +{n_0}^{\prime}-1=n_0-1$. This completes the proof on the condition
${\rm CohP}_{A,m^{\prime}+n_0-{n_0}^{\prime},n_0-1}$.
$D^{n_0^{\prime}+k} \times F$ can be immersed into ${\mathbb{R}}^{n_0+k}$ for any non-negative integer $k$ by virtue of the last condition (\ref{mthm:2.2.3}). 
$M^{\prime}$ satisfies the condition ${\rm Sp}_{\geq {n_0}^{\prime}}$. So we can take a special generic map $f_k$ on $M^{\prime}$ into ${\mathbb{R}}^{{n_0}^{\prime}+k}$ for $0 \leq k < m^{\prime}-{n_0}^{\prime}$. By considering the product map of such a special generic map (composed with an embedding into the interior of a copy of the disk $D^{n_0^{\prime}+k}$) and the identity map on $F$ and composing the immersion of $D^{n_0^{\prime}+k} \times F$, we have a special generic map from $M^{\prime} \times F$ into ${\mathbb{R}}^{{n_0}^{\prime}+k+n_0-{n_0}^{\prime}}={\mathbb{R}}^{n_0+k}$.
The dimension of $M^{\prime} \times F$ is $m^{\prime}+n_0-{n_0}^{\prime}=m$.
This completes the proof.
\end{proof}

For example, in Example \ref{ex:3}, let $m^{\prime}=2k$, ${n_0}^{\prime}=k+1$ and $n_0=2k+1$. 
We can take $l=1$ and $S^k \times S^k$, $S^k$, and
$S^k \times S^k \times S^k$ can be regarded as $M^{\prime}$, $F$, and $M^{\prime} \times F$, respectively, in Main Theorem \ref{mthm:2} (\ref{mthm:2.2}).

More generally, we can apply Main Theorem \ref{mthm:2} to connected sums and products of closed and connected manifolds inductively starting from standard spheres respecting conditions on dimensions of the closed and connected manifolds and the Euclidean spaces for example.


As $F$, we can take a compact and connected Lie group and the boundary of a closed tubular neighborhood of an embedding of a closed and connected manifold into a Euclidean space, for example.

Explicit observations on Main Theorem \ref{mthm:2} are left to readers and future problems for the author.

\section{Acknowledgment.}
The author would like to thank Osamu Saeki and Dominik Wrazidlo for discussions on the present study and encouragement. The author is a member of and supported by JSPS KAKENHI Grant Number JP17H06128 "Innovative research of geometric topology and singularities of differentiable mappings"(Principal investigator is Osamu Saeki). This work is supported by "The Sasakawa Scientific Research Grant".

\end{document}